\documentclass[a4paper, english, nolineno]{lipics-v2021-hacked}
\usepackage[utf8]{inputenc}
\usepackage{hyperref} 
\hypersetup{
    colorlinks=true,
    linkcolor=blue,
    filecolor=magenta,      
    urlcolor=cyan,
    citecolor=ForestGreen,
    pdftitle={Surfaces in Montesinos Knots},
    pdfpagemode=FullScreen,
    }
\usepackage{enumerate}
\usepackage[dvipsnames]{xcolor}
\usepackage{graphicx} 
\usepackage{textcomp} 

\DeclareMathOperator{\area}{area}
\DeclareMathOperator{\vol}{vol}

\newtheorem{algorithm}{Algorithm}

\title{Totally Geodesic Surfaces in Hyperbolic 3-Manifolds: Algorithms and Examples}

\hideLIPIcs

\titlerunning{Totally Geodesic Surfaces in Hyperbolic 3-Manifolds}
\author{Brannon Basilio}
{Dept. of Math., University of Illinois Urbana-Champaign, USA}{basilio3@illinois.edu}{}{}
\author{Chaeryn Lee}
{Dept. of Math., University of Illinois Urbana-Champaign, USA}{chaeryn2@illinois.edu}{}{}
\author{Joseph Malionek}
{Dept. of Math., University of Illinois Urbana-Champaign, USA}{jdm7@illinois.edu}{}{}

\authorrunning{B. Basilio, C. Lee, and J. Malionek}

\Copyright{Brannon Basilio, Chaeryn Lee, and Joseph Malionek}

\funding{All authors were supported by US National Science Foundation grants DMS-1811156 and DMS-2303572}

\supplement{Data and code that accompany this paper can be found in the Harvard Dataverse link here:  \href{https://doi.org/10.7910/DVN/3YGSTI}{https://doi.org/10.7910/DVN/3YGSTI}~\cite{DVN/3YGSTI_2024}}

\date{\today}

\newcommand{\Z}{\mathbb{Z}}
\newcommand{\R}{\mathbb{R}}
\newcommand{\C}{\mathbb{C}}
\newcommand{\Q}{\mathbb{Q}}
\newcommand{\pslc}{PSL(2, \C)}
\newcommand{\pslr}{PSL(2, \R)}
\newcommand{\slc}{SL(2, \C)}
\newcommand{\slr}{SL(2, \R)}
\renewcommand{\H}{\mathbb{H}}

\DeclareMathOperator{\tr}{tr}

\begin{document}

\maketitle

\begin{abstract}
    Finding a \textit{totally geodesic surface}, an embedded surface where the geodesics in the surface are also geodesics in the surrounding manifold, has been a problem of interest in the study of 3-manifolds.
    This has especially been of interest in hyperbolic 3-manifolds and \textit{knot complements}, complements of piecewise-linearly embedded circles in the 3-sphere.
    This is due to Menasco-Reid's conjecture stating that hyperbolic knot complements do not contain such surfaces.
    Here, we present an algorithm that determines whether a given surface is totally geodesic and an algorithm that checks whether a given 3-manifold contains a totally geodesic surface.
    We applied our algorithm on over 150,000 3-manifolds and discovered nine 3-manifolds with totally geodesic surfaces.
    Additionally, we verified Menasco-Reid's conjecture for knots up to 12 crossings. 
\end{abstract}

\section{Introduction}\label{Introduction}
Studying surfaces in 3-manifolds has been a theme since the field of 3-manifolds began.
Knowing the topology of these surfaces may give some information about the ambient 3-manifolds.
In particular, this paper focuses on hyperbolic 3-manifolds containing surfaces satisfying the property that the geodesic between any two points of the surface is the geodesic of those same two points when viewed as points in the 3-manifold.
These surfaces are called \textit{totally geodesic}.
For precise definitions and conventions please see Section~\ref{Background}.

One class of 3-manifolds we are interested in is \textit{knot complements}, complements of a properly embedded circle into $S^3$. 
Similarly, \textit{link complements} are complements of a disjoint union of properly embedded circles into $S^3$.
We explore surfaces in hyperbolic knots in the context of a conjecture of Menasco and Reid in~\cite{MenascoReid1992}: 

\begin{conjecture}\label{MenReid}
    Let $K$ be a knot in $S^3$ whose exterior $S^3\setminus K$ has a complete hyperbolic structure. Then $S^3 \setminus K$ does not contain a closed, embedded, totally geodesic surface.
\end{conjecture}
In reference to Conjecture~\ref{MenReid}, we assume throughout that all surfaces are closed and embedded, unless stated otherwise. 
Conjecture~\ref{MenReid} is known to hold for the following classes of knots: alternating knots \cite{MenascoReid1992}, Montesinos knots \cite{Oertel1984}, 3-bridge knots and double torus knots \cite{IchiharaOzawa}.
However, Conjecture~\ref{MenReid} does not extend to links, as there are links that contain such surfaces (see Figure 2 in \cite{MenascoReid1992}). For more details, see Appendix~\ref{menasco_reid_appendix}.

Although the context of the conjecture focuses on knot complements, we looked at many different censuses of 3-manifolds throughout this paper in order to find totally geodesic surfaces. 
We looked more generally at the class of hyperbolic manifolds called \textit{cusped hyperbolic 3-manifolds}.
These are manifolds which decompose into a compact 3-manifold with $n$ tori, $T^2$, boundary components and $n$ 3-manifolds each homeomorphic to $T^2 \times [1, \infty)$. 
In particular, we look at covers of small-volume manifolds and knot and link complements which have a hyperbolic structure.
While we focus on cusped hyperbolic manifolds, using the work of~\cite{Goerner2019VerifiedCF}, we believe that computations could be extended to include closed manifolds.
Now, let us define problems related to totally geodesic surfaces in hyperbolic 3-manifolds.

\medskip
\textsc{Totally Geodesic Surface}:

\textsc{Input:} An oriented hyperbolic cusped 3-manifold $M$, given as an ideally triangulated 3-manifold, a lift $\rho$ of its holonomy representation, and coordinates of a normal surface $F$.

\textsc{Output:} \textsc{Yes}, if $F$ is isotopic to a totally geodesic surface in $M$, \textsc{No} otherwise.

\medskip

Next, we extend this decision problem to an enumeration version.

\medskip

\textsc{Enumerate Totally Geodesic Surfaces}:

\textsc{Input:} An oriented hyperbolic cusped 3-manifold $M$, given as an ideally triangulated 3-manifold, along with a lift $\rho$ of its holonomy representation.

\textsc{Output:} The complete list of all normal surfaces in $M$ that are isotopic to totally geodesic surfaces, each given as a vector of normal coordinates.

\medskip

The algorithm which solves \textsc{Enumerate Totally Geodesic Surfaces} terminates in a finite number of steps by the argument found in Section~\ref{Algorithm} above Algorithm~\ref{finding_totally_geodesic}.

In Section~\ref{Algorithm} we produce several algorithms to solve these problems.
These algorithms come from the realization that with available tools like Regina \cite{regina}, finding embedded surfaces in 3-manifolds is computationally tractable, along with the fact that determining whether a surface is totally geodesic reduces to linear algebraic conditions \cite{Schultens2014IntroductionT3}.
Algorithm~\ref{skeletal_embedding} gives an embedding of the fundamental group of the surface into the fundamental group of the 3-manifold.
Algorithm~\ref{totally_geodesic} solves \textsc{Totally Geodesic Surface} and Algorithm~\ref{finding_totally_geodesic} solves \textsc{Enumerate Totally Geodesic Surfaces}.
In Section~\ref{Computations}, we do extensive computations on 150,000+ manifolds to demonstrate our algorithm and find previously unknown manifolds that contain totally geodesic surfaces, as well as give evidence for Conjecture~\ref{MenReid}.
Section~\ref{FutureWork} discusses other possible classes of surfaces to apply our algorithms to in order to extend our work.
The Appendices contain background and supplementary material as well as the proofs for Proposition~\ref{orientable_homology}, Lemmas~\ref{tot_geo_essential}, \ref{fundamental_skeleton_surface}, \ref{fundamental_skeleton_manifold}, \ref{skeletal_embedding_lemma}, Theorem~\ref{totally_geodesic_theorem} and Lemmas~\ref{VolumeGenus}, \ref{finite_isotopy_lemma}. 
Information about where to find our code and how to run it is also presented in the Appendix~\ref{code_availability}.

\section{Background}\label{Background}

In this paper, we assume that all 3-manifolds are oriented, hyperbolic, and cusped, with an ideal triangulation satisfying the gluing equations that give the unique complete hyperbolic structure.
We will explain these terms in detail in this section.
All surfaces considered will be closed and embedded.

\subsection{Hyperbolic 3-manifolds}\label{hyperbolic_background}

First, we give a brief introduction to the general theory of hyperbolic 3-manifolds along with details on some of the more specific computational background.
A complete background is outside of the scope of this paper but readers can find much of the theory in \cite{Purcell2020HyperbolicKT} and \cite{BenedettiPetronio}.

We start by reviewing some of the fundamentals of hyperbolic space and hyperbolic manifolds. 
Recall that \textit{$n$-dimensional hyperbolic space} $\H^n$ is the unique simply connected space of constant curvature $-1$. 
It can be modeled as the $n$-dimensional upper half-space, $\{\Vec{x}\in\R^n:x_n >0\}$ with the Riemannian metric determined by $ds = \frac{dx_1^2 + \dots + dx_n^2}{x_n^2}$.
With this model of $\H^n$, we define the \textit{boundary} of $\H^n$ to be $\{\Vec{x}\in\R^n:x_n =0\}\cup \{\infty\}$.
Note that topologically the boundary is homeomorphic to $S^{n-1}$.
The orientation-preserving isometries of $\H^n$, $\text{Isom}^+(\H^n)$ correspond bijectively to conformal (angle-preserving) maps on the boundary.
In particular, for $n=3$, if we identify the embedded plane $\{\Vec{x}\in\R^3:x_n =0\}$ with $\C$, the boundary then becomes the Riemann sphere and the orientation-preserving isometries of $\H^3$ can be identified with the group of M\"obius transformations which is isomorphic to $\pslc$.
An orientable hyperbolic 3-manifold $M$ is then a quotient of 3-dimensional hyperbolic space $\H^3$ by a torsion-free discrete subgroup $\Gamma$ of $\pslc$. 
This manifold has $\H^3$ as its universal cover with the deck transformations corresponding to the elements of $\Gamma$.
We say that $M$ is \textit{finite-volume} if the volume of $M$ with respect to the metric inherited on $M$ from $\H^3$ is finite.
Every finite-volume hyperbolic 3-manifold can be decomposed uniquely into a \textit{compact core} and some collection of \textit{cusps}~\cite{HarrisScott1996Core}.
The compact core consists of a compact manifold whose boundary is a possibly empty disjoint union of tori. 
The cusps consist of disjoint unions of $T^2\times [1,\infty)$. 
The two parts are glued together along their boundary tori to form $M$. 

\begin{figure}
	\centering
	\includegraphics[scale = 0.4]{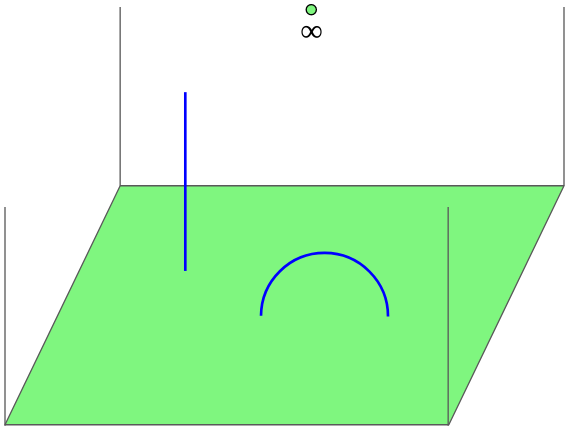}
	\caption{A model of $\mathbb{H}^3$. $\partial\mathbb{H}^3$ is given in green, while two examples of geodesics are given in blue.}
	\label{Dual2D}
\end{figure}

In order to work with manifolds concretely, it is often convenient to decompose them into simple pieces.
For cusped hyperbolic 3-manifolds, it is generally more convenient to use an \textit{ideal triangulation}.
Before that, we must introduce the concept of an ideal tetrahedron.
Topologically, an ideal tetrahedron is simply a tetrahedron with its vertices removed.
We can endow these with a hyperbolic geometric structure by taking an ideal tetrahedron to be the ``convex hull'' of any 4 distinct points in the boundary of $\H^3$.
The use of quotes here denotes the fact that we want to take the portion of that tetrahedron only in $\H^3$  and not on the boundary.
For every edge $e$ connecting the ``ideal'' vertices $a$ and $b$ of an ideal tetrahedron, we can find an orientation-preserving isometry which takes $a$ to $0$, and $b$ to $\infty$, and maps the remaining two ideal vertices to $1$ and some complex number $z$ lying in $\{x + iy\in\C: y \ge 0\}$. 
This number $z$ is called the \textit{shape parameter} of the edge of the tetrahedron.

\begin{figure}
	\centering
	\includegraphics[scale = 0.4]{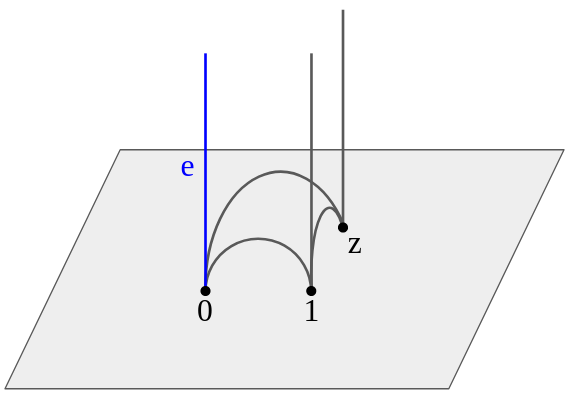}
	\caption{The shape parameter $z$ for the edge $e$ of the given ideal tetrahedron.}
	\label{Dual2D1}
\end{figure}

We mention here that for cusped hyperbolic manifolds $M$ we deal with throughout this paper, its ideal triangulation admits a partially flat angle structure as defined in~\cite{Lackenby2007AnAT}.

More information about how the shape parameters are used to build up the hyperbolic structure and the holonomy representation can be found in the Appendix~\ref{hyperbolic_background_appendix}.

\subsection{Totally Geodesic Surfaces}
A \textit{compressing disk} for a closed, connected, embedded surface $F$ in a 3-manifold $M$ is a disk $D \subset M$ such that $D \cap F = \partial D$ and $\partial D$ does not bound a disk in $F$.
A closed orientable surface is \textit{incompressible} if it does not have a compressing disk and is not a 2-sphere.
A closed orientable surface $F$ in $M$ is \textit{$\partial$-parallel} if $F$ is isotopic into $\partial M$.
A closed orientable surface is \textit{essential} if it is incompressible and is not $\partial$-parallel.

For a closed non-orientable surface $F$, we say that $F$ is \textit{incompressible} or \textit{essential} if the boundary of a regular neighborhood of $F$ has that property.
For a hyperbolic 3-manifold, an equivalent definition for incompressible is that the fundamental group of $F$, where $F$ is not a 2-sphere, injects into the fundamental group of $M$, i.e. $\pi_1(F) \hookrightarrow \pi_1(M)$ is injective.
We note that even for non-orientable $F$, $\pi_1(F) \hookrightarrow \pi_1(M)$ being injective implies that $F$ is incompressible.
Let $F$ be an embedded surface in $M$.
$F$ is said to be \textit{totally geodesic} if and only if every geodesic arc in $F$ with its induced Riemannian metric is also a geodesic in $M$.

Checking for totally geodesic surfaces is deeply related to the holonomy representation and we review the work of \cite{MenascoReid1992} for notions we use in our algorithms.
For a complete hyperbolic manifold $M$, there is a unique (up to conjugation) discrete, faithful representation $\rho:\pi_1(M)\to \pslc$ called the \textit{holonomy representation}.
This is the group of isometries used to construct the 3-manifold. It is often more convenient to work with a lift of the holonomy representation to $\slc$ which always exists by \cite{Culler_lifting86}.
For a surface $F$ embedded in a 3-manifold $M$, the inclusion map $i: F \to M$ induces a map $i_*: \pi_1(F)\to \pi_1(M)$ which composed with the holonomy representation, gives a representation from the fundamental group of $F$ to $\pslc$ (or $\slc$ when considering the lift of the holonomy representation).

The group of orientation-preserving isometries of $\H^3$ is $\pslc$.
Let $G \le \pslc$ be a discrete group and let $x \in \mathbb{H}^3$ be any point.
The \textit{limit set} of $G$ is the set of accumulation points on $\partial \mathbb{H}^3$ of the orbit $G(x)$.
A discrete subgroup of $\pslc$ is called a \textit{Kleinian group}.
A Kleinian group $\Gamma$ acting on the Riemann sphere $\C \cup \infty$ is \textit{quasi-Fuchsian} if its limit set is a Jordan curve.
$\Gamma$ is called \textit{Fuchsian} if the Jordan curve is a geometric circle.

A surface $F$ is called \textit{Fuchsian} if the induced representation $\rho(\pi_1 (F))$ is Fuchsian.
For an embedded non-orientable essential surface $F$ in $M$, if the induced representation $\rho(\pi_1 (F))$ is contained in the normalizer of $\pslr$ in $\pslc$ then $F$ is totally geodesic.
Note that for non-orientable surfaces, the property of being isotopic to a totally geodesic surface is equivalent to being Fuchsian.
This is not true for orientable surfaces since the double cover $\Tilde{F}$ of a non-orientable Fuchsian surface $F$ will also be Fuchsian, but $\Tilde{F}$ cannot be totally geodesic in $M$.
However, we have that if $F$ is Fuchsian, then either $F$ is isotopic to a totally geodesic surface or is a double cover of a non-orientable totally geodesic surface.
Moreover, there is a topological obstruction to the existence of essential non-orientable surfaces.

\begin{proposition}[Proposition 2.4 \cite{Dunfield_2022}]\label{orientable_homology}
    Suppose $M$ is a finite-volume orientable hyperbolic 3-manifold. Then every closed surface in $M$ is orientable if and only if the inclusion map $i:H_2(\partial M; \mathbb{F}_2)\to H_2(M;\mathbb{F}_2)$ is surjective.
\end{proposition}

Being totally geodesic implies Fuchsian, giving a necessary condition for a surface to be totally geodesic.
For totally geodesic surfaces, we also have the following well-known lemma.

\begin{lemma}\label{tot_geo_essential}
    A totally geodesic surface in a hyperbolic 3-manifold is essential.
\end{lemma}

\begin{proposition}[\cite{Purcell2020HyperbolicKT} Example 12.11]\label{Conjugate}
    Let $M$ be a complete hyperbolic 3-manifold with discrete, faithful representation $\rho:\pi_1(M) \to \pslc$ and $F$ a closed essential surface in $M$. If $F$ is orientable, then $F$ is Fuchsian if and only if $\rho(\pi_1 (F))$ can be conjugated into $\pslr \leq \pslc$. If $F$ is non-orientable, then $F$ is Fuchsian if and only if $\rho(\pi_1 (F))$ can be conjugated into the normalizer of $\pslr \leq \pslc$.
\end{proposition}

The above also holds for $\slc$. Since it is more convenient to work with matrices compared to projective matrices, we will deal with a lift of the holonomy representation into $\slc$ throughout this paper. Checking whether or not an $\slc$ representation conjugates into $\slr$ is difficult, so we reduce it to something simpler. 
\begin{lemma}[\cite{MorganShalen_part1} Proposition III.1.1]
    For an irreducible representation $\phi$ of the fundamental group $G$ of a surface into $\slc$ such that for all $g\in G$, $\tr(\phi(g))\in \R$, $\phi$ is either conjugate into $SU(2)$ or $\slr$.
\end{lemma}

If $F$ is a closed essential surface, then $\rho$ restricted to $\pi_1(F)$ is an irreducible representation.
Elements of $SU(2)$ as isometries of $\H^3$ all have fixed points so cannot represent elements of the fundamental group of $M$ and therefore cannot represent elements of the fundamental group of $F$. 
Hence, if $\rho$ restricted to $\pi_1(F)$ has all real traces, then it must be conjugate into $\slr$ and thus $F$ is totally geodesic or is a double cover of a totally geodesic surface.

The following lemma helps to simplify things further.
\begin{lemma}[\cite{MaclachlanReid} Lemma 3.5.3]\label{generator_G}
    For a subgroup $G$ of $\slc$ generated by $\{g_1,\dots,g_n\}$, the smallest field containing all of the traces of elements of $G$ is generated by the traces of the elements in $\{g_i, g_ig_j, g_ig_jg_k: 1 \le i\le j\le k \le n \}$. 
\end{lemma}

From the above, we get the following corollary.
\begin{corollary}\label{trace_generators}
    Let $M$ be a hyperbolic 3-manifold with discrete, faithful representation $\rho:\pi_1(M) \to \pslc$ and $F$ an essential orientable surface in $M$.
    Let $\{g_1,\dots,g_n\}$ be a set of generators of the subgroup $\rho(\pi_1 (F))$.
    Then $F$ is Fuchsian if and only if the traces of all the elements in the set $\{g_i, g_ig_j, g_ig_jg_k: 1 \le i\le j\le k \le n \}$ are real.
\end{corollary}

\subsection{Normal Surfaces}
In this section we review some basic normal surface theory following the work in~\cite{ojm/1200786486}. Let $M$ be a $3$-manifold with a triangulation $\mathcal{T}$. 
An \textit{elementary disk} $E$ in a tetrahedron $\Delta$ of $\mathcal{T}$ is a properly embedded disk that meets each edge of $\Delta$ in at most one point and each face of $\Delta$ in at most one line. 
A surface is said to be \textit{normal} if it is in general position with the 1-skeleton $\mathcal{T}^{(1)}$ of $\mathcal{T}$ and meets each tetrahedron only in elementary disks. 
We follow the convention that if $E \cap \mathcal{T}^{(1)}$ is a planar set then $E$ is planar and if not then $E$ is the cone $b \ast \partial E$ where $b$ is the centroid of the 3-simplex $\Delta$ spanned by $E \cap \mathcal{T}^{(1)}$. Elementary disks and hence normal surfaces are uniquely determined by their intersection points with $\mathcal{T}^{(1)}$.
A \textit{normal isotopy} of $M$ is an isotopy that leaves every simplex of $\mathcal{T}$ invariant.
There are exactly 7 normal isotopy classes of elementary disks: four triangles that each cut off a vertex and three quadrilaterals that separate the three pairs of disjoint edges.
We call such normal isotopy classes the \textit{disk types} of a tetrahedron.
Similarly, the normal isotopy classes of arcs t from the intersection of elementary disks and each $2$-face of $\Delta$ are called the \textit{arc types}.

\begin{figure}
	\centering
    \includegraphics[width=\textwidth]{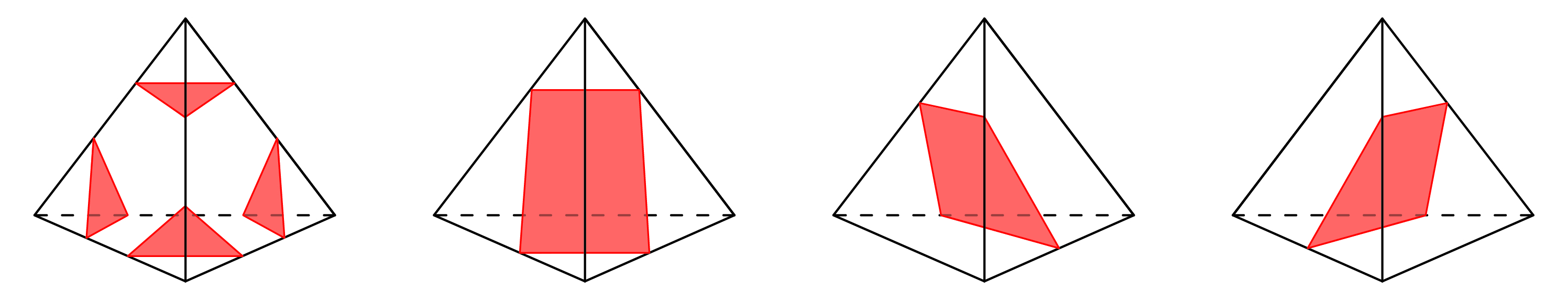}
	\caption{The 7 disk types.}
	\label{disktypes}
\end{figure}

The following demonstrates the generality of normal surfaces. 

\begin{theorem}[\cite{JACO1984195} Theorem 1.2, see Theorem 5.2.14 in \cite{Schultens2014IntroductionT3} for proof]\label{allsurface_iso_normal}
    Let $M$ be irreducible and let $S$ be a closed, incompressible surface in $M$. For any given triangulation $\mathcal{T}$ there is a normal surface $\mathscr{S}$ isotopic to $S$. 
\end{theorem}

The significance of understanding normal surfaces is that they can be expressed uniquely as a tuple of non-negative numbers.
Fix an ordering $d_1, d_2, \dots, d_{7t}$ on all disk types in $\mathcal{T}$ ($t$ is the number of tetrahedra in $\mathcal{T}$). To a normal surface $F$ we can assign a $7t$-tuple $\overrightarrow{F} = (x_1, x_2, \dots, x_{7t})$ where $x_i$ is the number of elementary disk types of type $d_i$ in $F$. $\overrightarrow{F}$ is called the \textit{normal coordinate} of $F$.
Any normal surface is uniquely determined by its normal coordinate up to normal isotopy.

A normal coordinate $\overrightarrow{x}$ is said to be \textit{admissible} if the corresponding normal surface has at most one nonzero quadrilateral disk type in every tetrahedron.
There is a one-to-one correspondence between all admissible solutions in the solution space of $\mathbb{R}^{7t}$ and normal surfaces in $M$.
More details about solution spaces are available in Appendix~\ref{normal_coords_appendix}.

\section{The Algorithms}\label{Algorithm}
In this section we present an algorithm that determines whether or not a 3-manifold contains a totally geodesic surfaces. 

\subsection{Computational Representations of Mathematical Objects}
We first introduce the specific computational descriptions of the mathematical objects we use in our algorithms in this section.

A cusped hyperbolic 3-manifold can be represented by an ideal triangulation.
This is a list of ideal tetrahedra, their shape parameters, and how each face of each tetrahedron is glued.
Fundamental groups of 3-manifolds and surfaces can be written using a specific group presentation: 
a set of generators $\{a_1,\dots, a_n\}$ for the group and a specific set of relations $\{r_1,\dots,r_m\}$ expressed as trivial words in the group.
A detailed treatment of generating sets for fundamental groups will be given in Section~\ref{CheckingTotGeoSec}.
Every fundamental group element can be represented as a word in the alphabet consisting of generators and their inverses.

The holonomy representation of an element of the fundamental group described in Section~\ref{hyperbolic_background} can be found from its group presentation and a set of matrices representing the generators as follows: any element $a_{i_1} a_{i_2} \cdots a_{i_k}$ in the group presentation is represented by the product of matrices that correspond to each of the generators $a_{i_1}, \dots, a_{i_k}$.
Due to the following result, holonomy representations of manifolds are of a certain form:

\begin{theorem}[Theorem 3.1.2 and Corollary 3.2.4 of \cite{MaclachlanReid}]
    For a finite-volume hyperbolic 3-manifold $M$, the holonomy representation $\rho$ can always be conjugated to lie in $PSL(2, \Q(\alpha))$ where $\alpha$ is some algebraic number.
\end{theorem}
This allows us to avoid some of the theoretical downsides of working with approximate numbers by using a representation of a number field in the following way, as found in \cite{SnapPaper}.
Given an arbitrary algebraic field extension $\Q(\alpha)$ where $\alpha$ is a root of the irreducible degree $n$ polynomial $p(x)\in\Q[x]$, every element $\beta\in\Q[x]$ can be uniquely represented as $c_0 + c_1\alpha + c_2\alpha^2 + \ldots + c_{n-1}\alpha^{n-1}$. 
More information about this way of representing algebraic field elements can be found in Appendix~\ref{field_appendix}.

In addition to fields, we also need information about specific \textit{embeddings} of a field into $\C$. 
Specifically, we need to determine whether a given embedding of a field into $\C$ is contained in $\R$. 
This can be done numerically as follows.
Every embedding of $\Q(\alpha)$ into $\C$ is completely determined by which root $\alpha_0$ of $p(x)$ $\alpha$ is sent to.
Therefore, it suffices to determine whether or not $\alpha_0$ is a real root of $p(x)$ or not.
Using Sturm's theorem or one of its generalizations (as in \cite{number_real_roots}), the number of real roots $r$ of $p(x)$ can be determined exactly.
The complex roots can be approximated by one's favorite convergent numerical method which will ensure that the root will not fall on the real line.
Once $n-r$ complex roots have been determined, the remaining roots must be real.

\subsection{Checking Whether a Surface is Totally Geodesic}\label{CheckingTotGeoSec}
We state two lemmas on finding the fundamental group of a space from a cellulation.

\begin{figure}
	\centering
	\includegraphics[scale = 0.4]{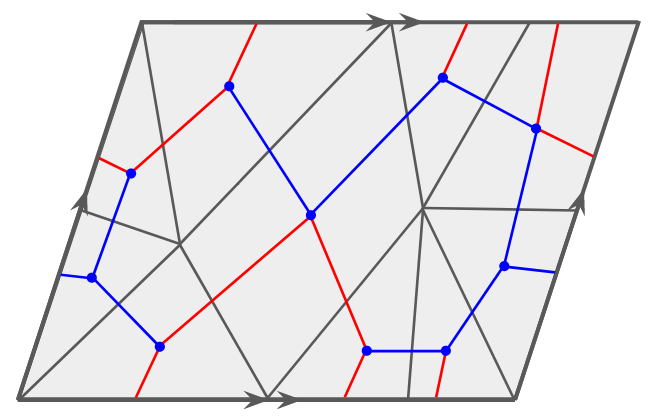}
	\caption{An example application of Lemma~\ref{fundamental_skeleton_surface} on a cellulation of a torus by polygons. The blue and red edges together form the dual 1-skeleton, with the blue edges making up the tree $\mathcal{T}$ and the red edges forming the generating set.}
	\label{Dual2D2}
\end{figure}

\begin{lemma}[\cite{EppsteinFundamentalGroup} Lemma 3.2]\label{fundamental_skeleton_surface}
    Assume that $F$ is a closed connected surface with a cellulation $\mathcal{C}$ by polygons. Then one can construct a 2-dimensional \textit{dual cellulation} $\mathcal{D}$ which is homeomorphic to $F$.
    Let $\mathcal{T}$ be a spanning tree of the 1-skeleton of $\mathcal{D}$. Then the edges in $\mathcal{D}\setminus \mathcal{T}$ form a generating set in some presentation of the fundamental group of $F$. 
\end{lemma}

\begin{figure}
	\centering
	\includegraphics[viewport = 10 90 650 750, scale = 0.35]{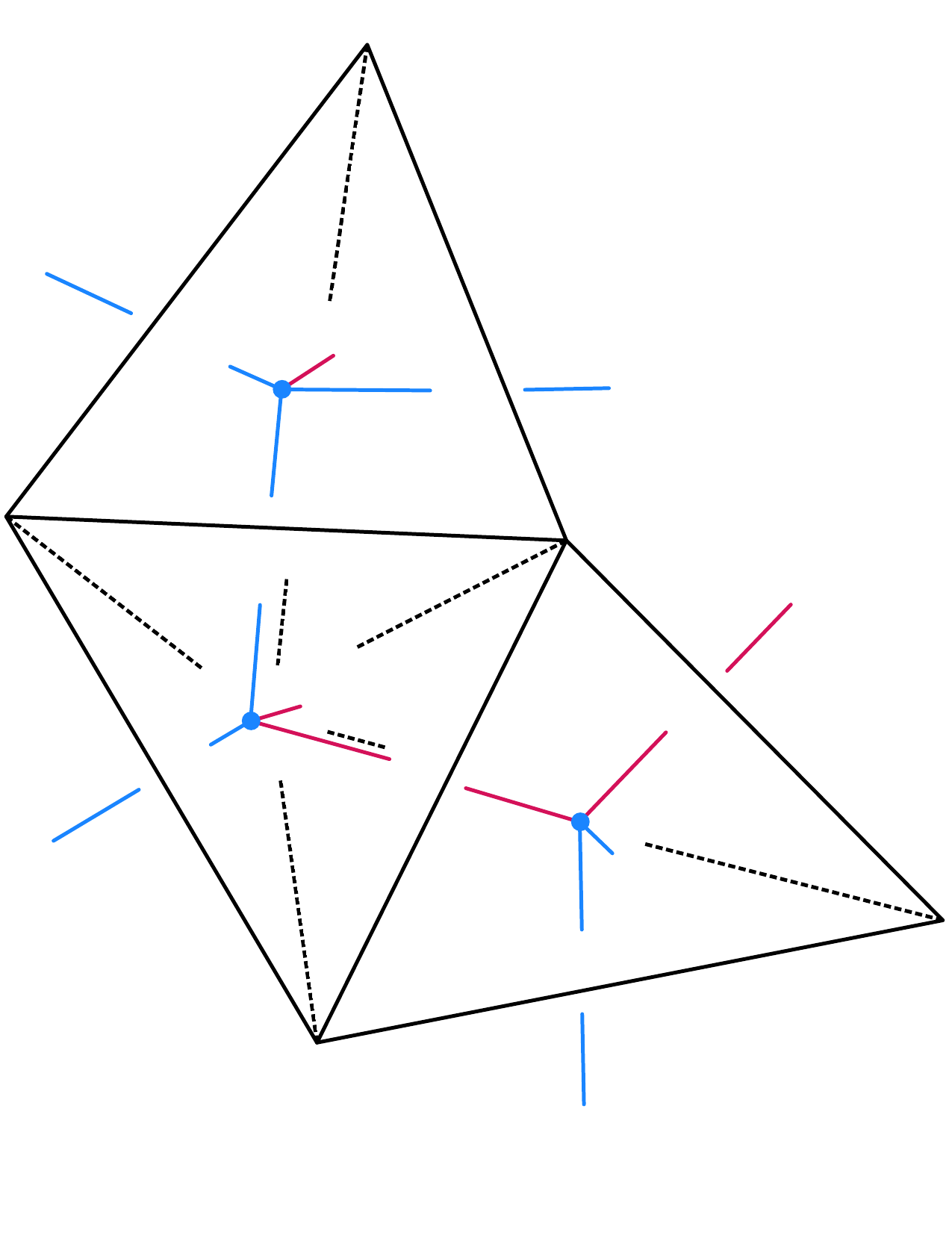}
	\caption{A partial drawing of a triangulated 3-manifold with its dual 1-skeleton as in Lemma~\ref{fundamental_skeleton_manifold}. The blue and red edges together form the dual 1-skeleton, with the blue edges belonging to the tree $\mathcal{T}$ and the red edges belonging to the generating set.}
	\label{Dual3D}
\end{figure}

\begin{lemma}\label{fundamental_skeleton_manifold}
    Assume that $M$ is a $3$-manifold with a cellulation $\mathcal{C}$ by ideal tetrahedra. Then one can construct a 2-dimensional \textit{dual cellulation} $\mathcal{D}$ which is homotopy equivalent to $M$.
    Let $\mathcal{T}$ be a spanning tree of the 1-skeleton of $\mathcal{D}$. Then the edges in $\mathcal{D}\setminus \mathcal{T}$ form a generating set in a presentation of the fundamental group of $M$.
\end{lemma}

\begin{figure}
	\centering
	\includegraphics[width=100mm]{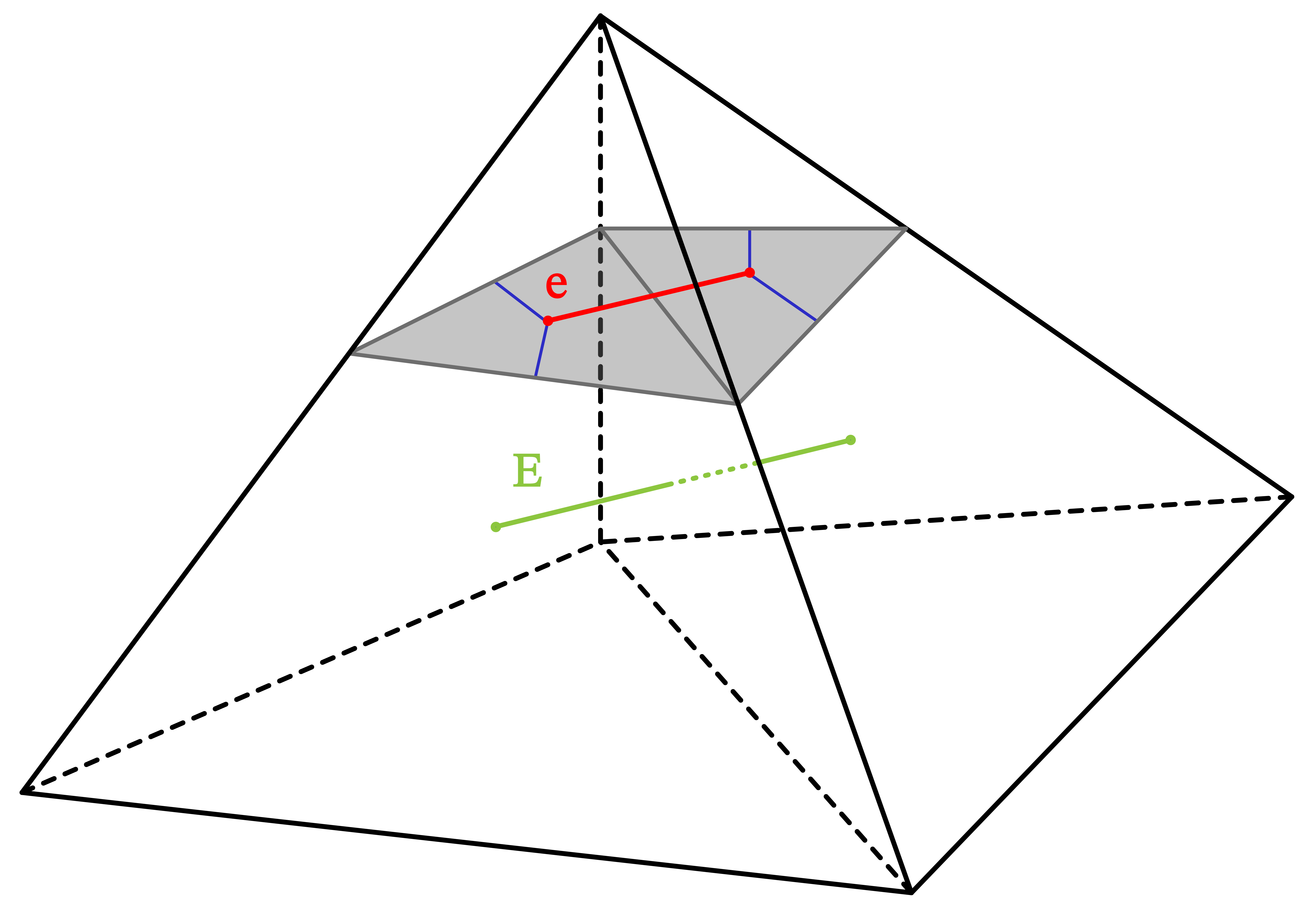}
	\caption{Suppose $F$ intersects the two tetrahedra in this figure in two triangles. $e$ is an element of the generating set of $\pi_1(F)$ and it corresponds to the edge $E$ in the dual 1-skeleton of $M$.}
    \label{stepc}
\end{figure}

We present algorithms that determine whether or not a given surface is totally geodesic. 

\begin{algorithm}\label{skeletal_embedding}
    Given a normal surface $F$ in a triangulated 3-manifold $M$, we get an embedding $i_* : \pi_1(F) \to \pi_1(M)$ by the following process:
    \begin{enumerate}[a)]
        \item Choose a basepoint for the fundamental group of $F$ in some normal disk in $F$.
        \item Obtain generating sets for $\pi_1(F)$ and $\pi_1(M)$ as in Lemmas~\ref{fundamental_skeleton_surface} and~\ref{fundamental_skeleton_manifold} with the spanning trees given by $\mathcal{T}_F$ and $\mathcal{T}_M$.
        \item Each element of the generating set for $\pi_1(F)$ corresponds to an edge $e$ in the dual 1-skeleton of $F$. The edge $e$ in turn, lies transversely to a face in the triangulation of $M$ which corresponds to an edge in the dual 1-skeleton of $M$. See Figure~\ref{stepc}.
        \item Get a loop in $F$ by taking a path in $\mathcal{T}_F$ from the basepoint disk to $e$ and then taking a path from $e$ back to the basepoint in $\mathcal{T}_F$. This gives a cycle $\gamma$ in the dual 1-skeleton of $F$.
        \item For each edge in $\gamma$, we get the corresponding edge in the dual 1-skeleton of $M$.
        \item The list of edges in the cycle $\gamma$ then correspond to an element of $\pi_1(M)$.
    \end{enumerate}
\end{algorithm}

\begin{lemma}\label{skeletal_embedding_lemma}
    Algorithm \ref{skeletal_embedding} terminates and the output is correct.
\end{lemma}

\begin{algorithm}\label{totally_geodesic}
    Given an essential normal surface $F$ in a triangulated 3-manifold $M$ and a lift of the holonomy representation of $\rho : \pi_1(M) \to SL(2,\Q(\alpha))$, the following algorithm solves \textsc{Totally Geodesic Surface}.
    \begin{enumerate}[a)]
        \item If $F$ is non-orientable, replace $F$ with $2F$, its orientable double cover, in the following.
        \item Compose the holonomy representation of $\rho : \pi_1(M) \to SL(2,\Q(\alpha))$ and the embedding $i_* : \pi_1(F) \to \pi_1(M)$ found in Algorithm~\ref{skeletal_embedding}. 
        \item Take the generators found in Lemma \ref{fundamental_skeleton_surface} as matrices in $SL(2,\Q(\alpha))$. 
        \item Construct the smallest field extension $\Q(\beta)$ (as a subfield of $\Q(\alpha)$) containing the traces of all generators and all products of pairs and triples of generators, using results from Section~4.5 of \cite{Cohen1993AlgebraicNT}.
        By Lemma~\ref{generator_G} this field will contain the traces of all elements in $\pi_1(F)$.
        \item Check whether the embedding of $\Q(\beta)$ into $\C$ is real. If not, the surface $F$ cannot be totally geodesic. Output \textsc{No}.
        \item If $F$ was originally non-orientable, then $F$ is totally geodesic. Output \textsc{Yes}.
        \item The only case that remains is if $F$ is orientable and the embedding of $\Q(\beta)$ into $\C$ is real. We now need to check that $F$ is not the double cover of a non-orientable surface. Let $\chi(F) = X$. 
        First enumerate all normal surfaces of Euler characteristic $X/2$.
        From Theorem 6.9 of \cite{Dunfield_2022} we can find all isotopy classes of essential normal surfaces of Euler characteristic $X$.
        Check if there is any normal surface in the isotopy class of $F$ that is a double of some normal surface of Euler characteristic $X/2$.
        \item If yes, $F$ can be isotoped to a double of a surface and $F$ is not totally geodesic (but it is the double cover of a totally geodesic surface). Output \textsc{No}.
        \item If the code has made it to this step, then $F$ is Fuchsian and is not the double cover of a non-orientable totally geodesic surface and is thus totally geodesic. Output \textsc{Yes}.
    \end{enumerate}
\end{algorithm}

\begin{theorem} \label{totally_geodesic_theorem}
    Algorithm \ref{totally_geodesic} produces a correct output.
\end{theorem}

\begin{proof}
    This follows from Corollary \ref{trace_generators}.
\end{proof}

\subsection{Detecting Totally Geodesic Surfaces in a 3-manifold}
Now given a 3-manifold, we describe an algorithm that checks whether it contains a totally geodesic surface. The idea is to enumerate all normal surfaces in the 3-manifold and implement the algorithm in the previous section. The 3-manifold and normal surfaces should first satisfy certain volume constraints. 
By a result of Miyamoto (Theorem 4.2 \cite{MIYAMOTO94}), we obtain a lower bound on the volume of all 3-manifolds containing a closed totally geodesic surface.
Moreover, if a given 3-manifold $M$ contains such a surface, the result also gives an upper bound on the Euler characteristic of a totally geodesic surface.

\begin{lemma}\label{VolumeGenus}
    Let $M$ be a hyperbolic 3-manifold. If $M$ contains an embedded closed totally geodesic surface $F$, then 
    \begin{equation*}        
        \vol M \geq 2 \mu_3(0) \area  F = 4 \pi \mu_3(0) \lvert \chi(F) \rvert,
    \end{equation*}
    where $\mu_3(0) \approx 0.29156$ as in \cite{MIYAMOTO94}. For orientable surfaces $F$ we have that $\vol M \geq 2 \mu_3(0) \area  F \gtrapprox 7.3277$ and for non-orientable $F$ we have that $\vol M \geq 2 \mu_3(0) \area  F \gtrapprox 3.66385$.
    Moreover, if a given 3-manifold M contains a closed totally geodesic surface $F$, then the Euler characteristic of $F$ is at most $\frac{\vol M}{4\pi \mu_3(0)}.$
\end{lemma}

It remains to find all surfaces in a given 3-manifold $M$ that satisfy this volume bound to check if it is totally geodesic.

\begin{lemma}\label{finite_isotopy_lemma}
For a triangulated 3-manifold $M$ and some number $e$, there is a finite number of isotopy classes of essential surfaces with Euler characteristic larger than $e$.
\end{lemma}

Upon finding all normal surfaces, by using Theorem 4.2 of~\cite{JACO1984195} we can determine whether each normal surface is essential in order to obtain a finite list of all essential normal surfaces up to a certain Euler characteristic. 

\begin{algorithm}\label{finding_totally_geodesic}
    Given a triangulated 3-manifold $M$ and a lift of the holonomy representation of $\rho : \pi_1(M) \to SL(2, \Q(\alpha))$, there is an algorithm to enumerate the totally geodesic surfaces in $M$, solving \textsc{Enumerate Totally Geodesic Surfaces}.
    \begin{enumerate}[a)]
        \item If $\text{vol }M < 3.66385$, then return false.
        \item Enumerate all isotopy classes of essential surfaces of $M$ which satisfy the Euler characteristic bounds of Lemma~\ref{VolumeGenus}.
        \item Apply Algorithm~\ref{totally_geodesic} to all surfaces found in the previous step.
    \end{enumerate}
\end{algorithm}

\section{Computations}\label{Computations}

\subsection{Practical Algorithm Considerations}    
Our computations were all performed in SageMath~\cite{sagemath} on KEELING, the School of Earth, Society \& Environment computer cluster at the University of Illinois Urbana-Champaign.
It made use of the programs Regina~\cite{regina} and SnapPy~\cite{SnapPy}. In particular, all surface enumeration was done by Regina based on the work in~\cite{Burton_2013}.
The data sets which we ran our algorithm on come from a variety of censuses in SnapPy. 
More details for these censuses can be found in Section~\ref{Results}.
Geometric information about the manifolds included in these censuses were provided by SnapPy.

\begin{figure}
	\centering
	\includegraphics[scale=0.7]{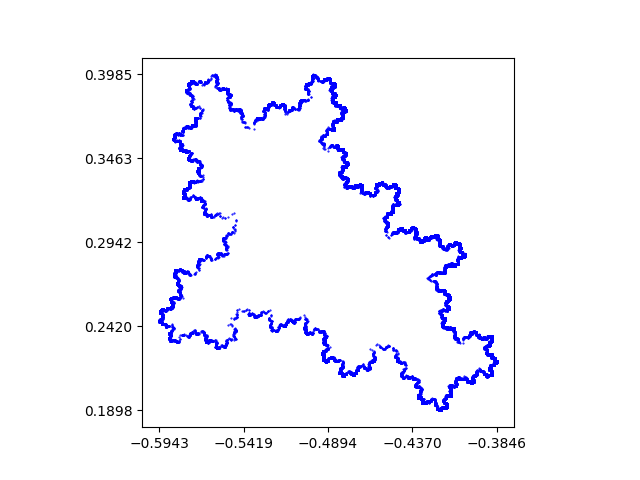}
	\caption{Approximation of the limit set of a quasi-Fuchsian surface that is not totally geodesic. }
	\label{quasifuchsian}
\end{figure}

While the algorithms specify an exact solution to the gluing equations, finding and working with an exact form of the holonomy representation is computationally expensive. 
To speed up and simplify our computations, we instead chose to work with double precision and then check that the traces are real by ensuring that the imaginary parts are less than a specified threshold. 
A middle ground between these two approaches would be to use interval arithmetic along with verified computations of the holonomy representation as found in \cite{Dunfield_Lspace14}.
This would have the advantage of being able to provably rule out any non-totally geodesic surfaces with only very minor performance losses.
However, we opted for double precision because the observed values used to rule out the surfaces were high enough to be unambiguous (nearly always at least 1).
With the large proportion of negative results we expected our code to find, we chose the threshold to be quite large (.01) so that if the algorithm ruled out the surface we could be confident that it did so correctly. 
However, even with this large threshold, it is still unlikely for the algorithm to return a false positive, given that the trace of \textit{every} checked matrix (and recall that the number of these matrices is cubic in the number of generators of the fundamental group of the surface) would have to have its imaginary part below this threshold.
Because of this high threshold, we still decided to do further checks to ensure that every surface the algorithm marked as positive was, in fact, totally geodesic. Some additional sanity checks on our code are also detailed in Appendix~\ref{sanity_checks}

\begin{figure}
	\centering
    \begin{subfigure}[b]{0.43\textwidth}
        \includegraphics[width=\textwidth]{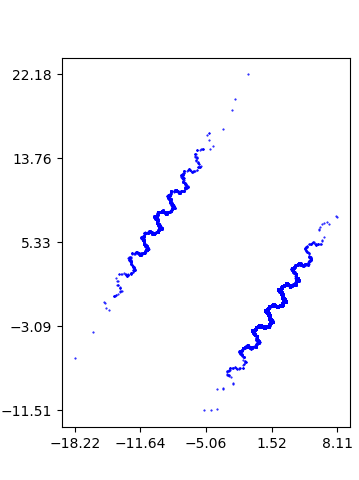}
    \end{subfigure}
    \hfill
    \begin{subfigure}[b]{0.54\textwidth}
        \includegraphics[width=\textwidth]{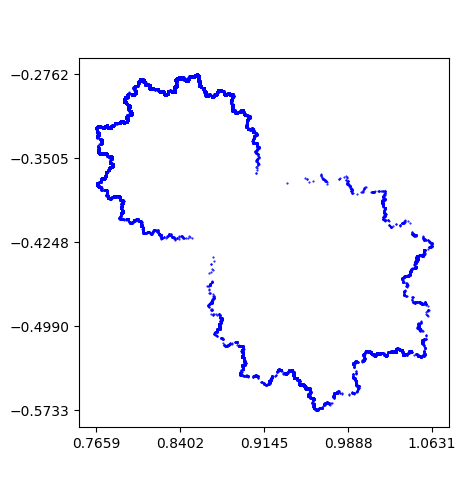}
    \end{subfigure}
	\caption{Approximations of limit sets of surfaces with an accidental parabolic.}
	\label{accidentalparabolic}
\end{figure}

To confirm that a surface identified as possibly totally geodesic is indeed totally geodesic, we plotted the limit set of the Kleinian group $\Gamma$ corresponding to the surface's fundamental group.
Applying an isometry corresponding to a long geodesic in the surface to any point in $\partial \H^3$ will result in a point very close to one of the ends of the geodesic on the boundary. In order to get these points, we need only find the isometries corresponding to long words in the fundamental group of $M$ and then apply them on any point in $\partial \H^3$.
As mentioned in Section~\ref{Background}, if the surface is totally geodesic, then the limit set of $\Gamma$ must be a geometric circle, not just a topological one.
We present some examples of approximations of limit sets of surfaces that are not totally geodesic. 
Figure~\ref{quasifuchsian} is the limit set of an essential quasi-Fuchsian surface that is not totally geodesic found in~\cite{adams_reid_1993}. In particular, it is a surface of genus 2 in the exterior of the knot in Figure 7(a) of~\cite{adams_reid_1993}. Figure~\ref{accidentalparabolic} are limit sets of essential surfaces in the exterior of alternating knots. They are surfaces of genus 2 found in the exteriors of knots $8_{16}$ and $9_{29}$, names for these knots follow the online database, \textit{KnotInfo:Table of Knots}~\cite{knotinfo}. These surfaces are all known to have an accidental parabolic. 
Details on the parameters used to plot these approximations is given in the description of Figure~\ref{limit_set} of Section~\ref{Example}.

It should be noted that some manifolds took significantly longer amounts of time to run the program on than other manifolds. 
Thus, in order to save time and increase the quantity of manifolds computed, the program terminated for manifolds after exceeding a certain chosen runtime. 
This runtime was chosen to be 5,000 seconds, based on the average runtime of census manifolds.
To ensure the algorithm ran on all knot exteriors with at most 12 crossings, a few manifolds whose runtime exceeded 5,000 were allowed but their data is not included in the average calculation or plots in Section~\ref{Results}.

\subsection{Results}\label{Results}
We observed a variety of 3-manifolds provided by databases from SnapPy. 
We mainly looked at link exteriors and covers of orientable cusped hyperbolic manifolds. 
The knot and link exteriors that were considered came from the census of 15 crossing knots and links of Hoste-Thistlethwaite-Weeks \cite{HTLinkExterior}.
We only take link exteriors in $S^3$ with volume larger than 7.2 (Lemma~\ref{VolumeGenus}) as these link exteriors do not contain non-orientable essential surfaces since $S^3$ does not have such surfaces. 
Moreover, we exclude alternating links~\cite{MenascoReid1992} where the result is already known, to get that the list contained 279,649 manifolds.
The triangulation information we used came from the census \textit{HTLinkExteriors} in SnapPy.

\begin{figure}
		\centering
	\centering
    \begin{subfigure}[b]{0.7\textwidth}
        \includegraphics[width=\textwidth]{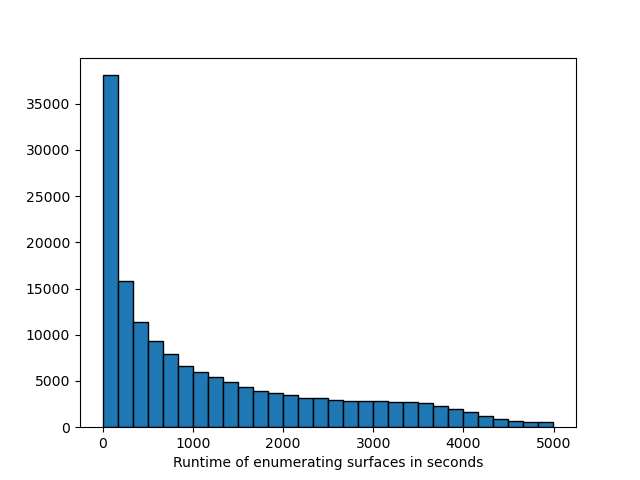}
    \end{subfigure}

    \begin{subfigure}[b]{0.7\textwidth}
        \includegraphics[width=\textwidth]{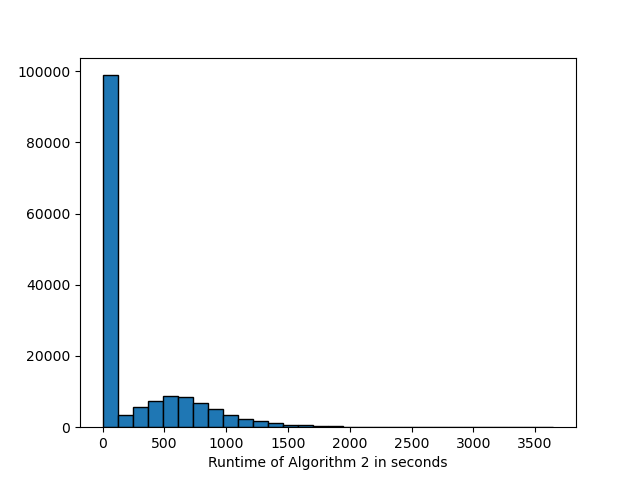}
    \end{subfigure}
	\caption{The top figure is a histogram of runtimes for enumerating surfaces measured in seconds. The bottom one is a histogram of runtimes for running Algorithm~\ref{totally_geodesic}.}
\label{runtime_histogram}
\end{figure}

The second census that was considered are covers of 3-manifolds provided by the \textit{OrientableCuspedCensus} in SnapPy~\cite{Burton2014TheCH}. 
This census consists of orientable cusped hyperbolic 3-manifolds that can be triangulated with at most 9 ideal tetrahedra.
It contains 61,911 manifolds, some of which contain multiple cusps.
To give ourselves plenty of examples to work with we took $n$-fold covers of these manifolds where $n$ allows the volume to be large enough to admit a totally geodesic surface but small enough so that the volume of the resulting cover is at most 20 (at this volume, our algorithm's runtime starts to get prohibitively long).

\begin{figure}
		\centering
		\includegraphics[scale = 0.7]{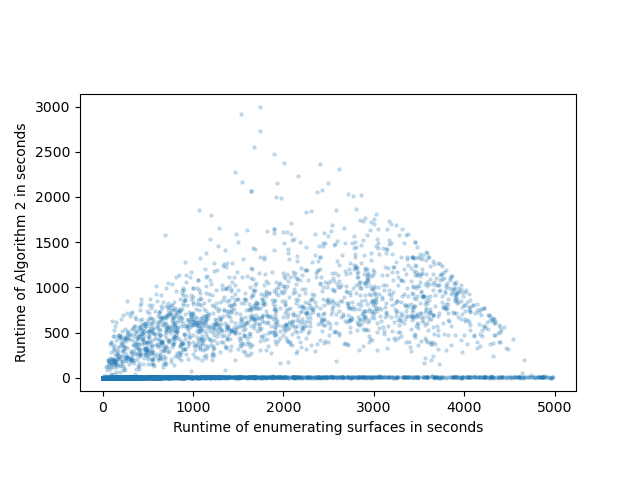}
		\caption{Scatter plot of the two different runtimes in seconds of 5000 randomly chosen manifolds.}  
		\label{runtime_comparison}
\end{figure}

We have currently completed computation for 142,409 manifolds in the \textit{HTLinkExteriors} census in SnapPy and 15,992 covers of manifolds in the \textit{OrientableCuspedCensus} in SnapPy. 
The average total runtime was 1483.37 seconds.

Figure~\ref{runtime_histogram} are histograms of the runtimes of manifolds for which we completed our computation.
We have made the distinction between the time it takes to enumerate surfaces in a manifold and applying Algorithm~\ref{totally_geodesic} to these found surfaces.  
On average, enumerating surfaces constituted 87\% of the the entire runtime whereas applying Algorithm~\ref{totally_geodesic} constituted 13\%.
Since for small manifolds the runtimes for both procedures are very small and have almost no difference this average can be a somewhat coarse indicator.
A more detailed comparison is given in Figure~\ref{runtime_comparison}.
Notice in Figure~\ref{runtime_comparison} that in general, runtimes for enumerating surfaces are relatively larger than the runtime for applying Algorithm~\ref{totally_geodesic}.
It is also interesting to examine results regarding runtimes of enumeration of normal surfaces in~\cite{BurtonComplexity}. In the worst case, finding all normal surfaces below the given genus bound of Algorithm~\ref{finding_totally_geodesic} is known to scale exponentially with the volume of the manifold.
Even in the average case, there is evidence that the number of normal surfaces scales exponentially with the number of tetrahedra, limiting the size of manifolds we can run the algorithm on.

\begin{figure}
		\centering
		\includegraphics[scale = 0.7]{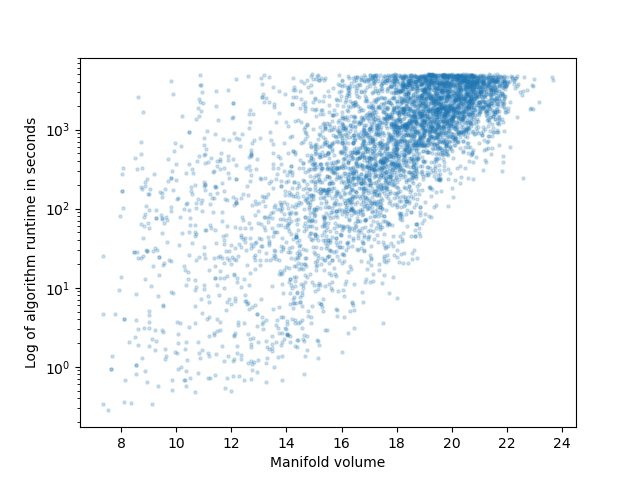}
		\caption{Scatter plot of the volume and log of runtime taken for entire algorithm. A random 5000 manifolds were chosen to be plotted.}
		\label{runtime_volume}
\end{figure}

\begin{figure}
		\centering
		\includegraphics[scale = 0.7]{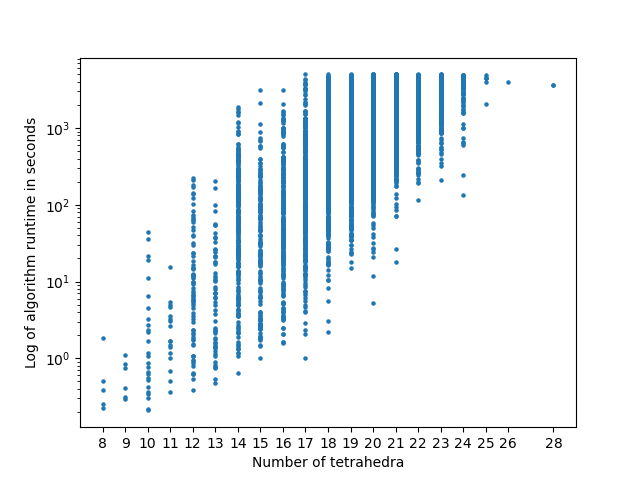}
		\caption{Scatter plot of the number of tetrahedra and log of runtime taken for entire algorithm. Again a random 5000 manifolds were chosen to be plotted.}
		\label{runtime_num_tet}
\end{figure}

We speculate that the number of tetrahedra and volume of the manifold have the biggest effect on runtimes.
Relations between runtimes of our algorithm and number of tetrahedra and volume are presented in Figures~\ref{runtime_volume} and~\ref{runtime_num_tet} respectively.
We chose to plot the log of runtimes to make the relation more explicit.
Moreover, in Figures~\ref{runtime_comparison},~\ref{runtime_volume} and~\ref{runtime_num_tet}, manifolds where no surfaces were detected (hence Algorithm~\ref{totally_geodesic} did not run at all) were not plotted. 

Determining whether or not a given normal surface is essential can be an immensely time consuming computation.
Algorithm~\ref{totally_geodesic} usually takes much less time than checking if a surface is essential, hence we chose to run our algorithm on all normal surfaces only excluding the simple compressible surfaces identified by the method provided by Letscher explained in Appendix~\ref{implementation_details}.
For surfaces found to be totally geodesic, we check their limit sets to make sure that the surface found was indeed totally geodesic or homotopic to a totally geodesic surface.
This surface could be homotopic to a totally geodesic surface because for any totally geodesic surface $F$ (which would be essential by Lemma~\ref{tot_geo_essential}) we can add a trivial handle to that surface to create a non-essential and non-totally geodesic surface $F'$. 
However, our algorithm would still detect $F'$ as a totally geodesic surface because the image of the fundamental group through the holonomy representation would be the same as $F$.
If our algorithm does find such an $F'$, we do know that there exists some $F$ in $M$ which is totally geodesic.

\begin{table}
\centering
{\small
\begin{tabular}{|c|c|c|}
\noalign{\smallskip}\noalign{\smallskip}\hline
Base manifold & Degree of covering & $H_1$ \\
\hline
m003(0,0) & 9 & $\Z_{10} \oplus \Z_{10} \oplus \Z$ \\
m004(0,0) & 9 & $\Z_{4} \oplus \Z_{4} \oplus \Z$ \\
m412(0,0)(0,0) & 3 & $\Z_{2} \oplus \Z_{2} \oplus \Z^2$ \\
s594(0,0) & 3 & $\Z_{2} \oplus \Z_{2} \oplus \Z_{4} \oplus \Z$ \\
s594(0,0) & 4 & $\Z_{2} \oplus \Z_{2} \oplus \Z^4$ \\
s596(0,0)(0,0)* & 3 & $\Z_{2} \oplus \Z^4$ \\
s955(0,0) & 3 & $\Z_{20} \oplus \Z$ \\
s956(0,0) & 3 & $\Z_{2} \oplus \Z_{4} \oplus \Z$ \\
s957(0,0) & 3 & $\Z_{5} \oplus \Z_{20} \oplus \Z$ \\
\hline

\end{tabular}
\caption {List of manifolds which were covers of manifolds in the OrientableCuspedCensus that contained totally geodesic normal surfaces. In this table, the information is enough to uniquely identify these manifolds (note however, that the starred entry needs the additional piece of information that its symmetry group is the product of $\Z_2$ with the octahedral group.)}
\label{table:occ_cover_list}
}
\end{table}

In Table~\ref{table:occ_cover_list} we present a list of manifolds that we found to contain a totally geodesic surface. 
All manifolds were covers of a manifold in the \textit{OrientableCuspedCensus} provided by SnapPy. 
Through private correspondence with Nathan Dunfield, he has informed us that all of these manifolds are commensurable to the figure-8 knot complement \textit{m004} which, as mentioned in Appendix~\ref{menasco_reid_appendix}, is known to have infinitely many \textit{immersed} closed totally geodesic surfaces.

It is worth mentioning that of the 142,409 link exteriors we ran the algorithm on, none contained totally geodesic surfaces. 
Our algorithm finished running for all non-alternating knot exteriors of up to 12 crossings and we anticipate more to be completed as our computations are ongoing. 
These results give strong support for Menasco and Reid's conjecture.

\subsection{An Example Containing a Totally Geodesic Surface}\label{Example}
Here we specifically look in detail at the last manifold in Table~\ref{table:occ_cover_list}, the 3-fold cover of the census manifold \textit{m412(0,0)(0,0)} which we will call $Y$. The triangulation information we used for our computation is recorded as a tight encoding of a Triangulation3 class in Regina:
\begin{equation*}
\verb|1%")"-"(*&"*"+")*\|\text{\textquotesingle}\verb|"(","**/".,+,0"/,-,."0,,,."-*/",*0"+*.*/*0*|
\end{equation*}
This triangulation had 15 tetrahedra and its volume was $15.2241240961448$. 
Algorithm~\ref{finding_totally_geodesic} found one potential totally geodesic surface which we will call $S$. $S$ was a surface of genus $2$ whose corresponding normal coordinate $\overrightarrow{S}$ is given as the following vector:
\begin{equation*}
\begin{aligned}
\overrightarrow{S} = &  (0, 0, 0, 2, 0, 0, 0, 0, 2, 0, 0, 0, 0, 0, 2, 0, 0, 0, 0, 0, 0, \\
& 0, 0, 2, 0, 0, 0, 0, 0, 2, 0, 0, 0, 0, 0, 2, 0, 0, 0, 0, 0, 0, \\
& 2, 0, 0, 0, 0, 0, 0, 0, 0, 2, 0, 0, 0, 0, 0, 2, 0, 0, 0, 0, 0, \\
& 0, 0, 0, 0, 0, 2, 0, 0, 0, 0, 0, 0, 0, 2, 0, 0, 0, 0, 2, 0, 0, \\
& 2, 0, 0, 0, 0, 0, 0, 0, 0, 0, 2, 0, 0, 0, 0, 2, 0, 0, 0, 0, 0) \\
\end{aligned}
\end{equation*}
 
\begin{figure}
		\centering
		\includegraphics[scale = 0.7]{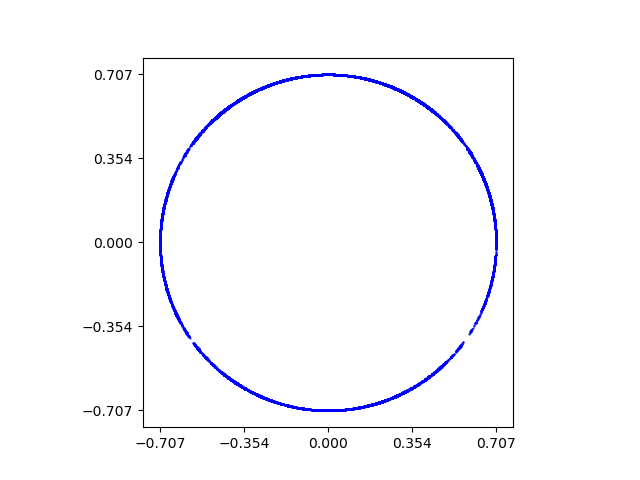}
		\caption{Limit set of the surface $S$.}
		\label{limit_set}
\end{figure}

Let $\Gamma$ denote the fundamental group representation of $S$ into $\pslc$ using the holonomy representation of $Y$ as in Algorithm~\ref{finding_totally_geodesic}. An approximation of the limit set of $S$ is presented in Figure~\ref{limit_set}. All limit set approximations in this paper consist of $10,000$ points obtained by applying up to $2000$ randomly chosen matrices in $\Gamma$ to the complex point $(1, 0)$. It is evident from Figure~\ref{limit_set} that the approximation resembles a geometric circle. 

We also computed the trace field of $\Gamma$. To compute the trace field of $\Gamma$, we computed the traces of all generators and products of pairs and triples of all of the generators. Using SageMath, we calculated the smallest field containing these traces. The result was that the trace field of $\Gamma$ is $\Q(\sqrt{3})$. The traces of all elements of $\Gamma$ are real and the surface $S$ is Fuchsian. Note that all coordinates of the vector of $S$ are even. Hence $S = 2F$ for some non-orientable surface $F$. By above, we have that $F$ is totally geodesic while $S$ is just Fuchsian.

\section{Future Work}\label{FutureWork}

One avenue we hope to expand our work is to enlarge the class of surfaces to look at.
One possibility could be to examine embedded totally geodesic surfaces with boundary.
There are known embedded totally geodesic surfaces with boundary, for example \textit{Seifert surfaces} in a knot complements (see \cite{Adams_Schoenfeld05}) and thrice-punctured spheres in hyperbolic 3-manifolds (see \cite{Adams1985}).
It would be interesting to determine if there are additional examples of embedded totally geodesic surfaces with boundary which are not Seifert surfaces nor thrice punctured spheres.

Using the idea of \textit{spun normal surfaces}, this would be a natural and reasonable extension.
Spun normal surfaces are a generalization of normal surfaces due to unpublished ideas of Thurston in the context of hyperbolic 3-manifolds with cusps which allows for infinitely many triangles in tetrahedra.
These infinitely many triangles would not be a problem as the surfaces generally deformation retract onto a subcomplex with only finitely many triangles \cite{Tillmann2004NormalSI}.
The key here is that the number of quads is always finite, and can be solved for using a similar set of equations to regular normal surfaces.
So after enumerating the normal surfaces in this way, we can find finite cellulations of them by quads and triangles and the remainder of the algorithm will work without much modification.

We also plan to extend our current work on plotting limits sets of surfaces and build on the idea of Thurston. It is known that one can determine whether or not a surface is essential by looking at the shape of its limit set. Utilizing our method may produce a computationally efficient and practical algorithm that verifies if a given normal surface is essential.

\bibliographystyle{plainurl}
\bibliography{bib}

\appendix
\section{Background}\label{background_appendix}

\subsection{Hyperbolic Structures and the Holonomy Representation}\label{hyperbolic_background_appendix}
Recall that for an ideal tetrahedron, you can specify an embedding of this ideal tetrahedron into hyperbolic space by using its shape parameters.
Shape parameters are important because they allow us to computationally determine whether a given ideal triangulation can be given a hyperbolic structure following the procedures in \cite{WeeksHyperbolic}.
In order to build a hyperbolic structure on $M$ from an ideal triangulation, all that needs to be done is to choose shape parameters which satisfy two types of criteria, the edge criteria and the cusp criteria. 
The edge criteria states that as one glues the tetrahedra, face by face, around each edge, the last face glues up correctly with the first face. 
The second type of criterion, one for each cusp, states that each cusp will end up being an honest-to-god euclidean torus with a complete structure, that is, a torus which tiles the euclidean plane.
We omit the details here, but both of these criteria can be stated as simple linear equations in the logs of the shape parameters.
Furthermore, it can be arranged for this system to reduce to a balanced system, ensuring that if a solution exists, it will be unique.

Given an ideally triangulated hyperbolic 3-manifold $M$, shape parameters for the tetrahedra, and an element of the fundamental group $\gamma$, we can find the holonomy representation in the following way using the idea of the \textit{developing map} due to Thurston in \cite{Thurston1979TheGA}.
Let $\gamma$ be a closed curve in $M$.
Isotope $\gamma$ so that it is transverse to the 2-skeleton of $M$.
Then we can find a corresponding list of tetrahedra $(\tau_1, \tau_2, ..., \tau_n)$ that $\gamma$ goes through.
Place the first tetrahedron $\tau_1$ in $\mathbb{H}^3$ so that it is of the correct shape. 
Then place subsequent copies of tetrahedra $\tau_i$ in this list so that they are glued appropriately until the entire list has been exhausted. 
The projective transformation taking the initial tetrahedron to the final tetrahedron can then be represented by a matrix in $\pslc$ and will be the holonomy representation of the element of the fundamental group.
Note that this only needs to be done for a generating set to get complete information about the holonomy representation.

\subsection{Conjecture~\ref{MenReid} for Links}\label{menasco_reid_appendix}
In the example of \cite{MenascoReid1992}, the totally geodesic surface separates components of the link.
Leininger in \cite{Leininger2006} constructs an infinite family of hyperbolic link complements each containing a totally geodesic surface $F$ which has the property that $F$ separates $S^3$ into two components, each containing the same number of components of the link.
It is noticed in \cite{Leininger2006} that the techniques of Adams~\cite{Adams1985} yields totally geodesic surfaces in link complements that separate distinct nonzero numbers of components of the link.
DeBlois in \cite{DeBlois2006} showed that there exist hyperbolic knots in rational homology spheres containing such surfaces.
Moreover, Reid~\cite{reid_1991} showed that \textit{immersed} closed totally geodesic surfaces exist in the figure-8 knot complement and Aitchison and Rubinstein~\cite{AitchRub97} computationally found such surfaces in two other knot complements.

\subsection{Normal Coordinates and Solution Spaces}\label{normal_coords_appendix}
Let $\overrightarrow{x} = (x_1, x_2, \dots, x_{7t})$ be some $7t$-tuple of non-negative integers.
For $\overrightarrow{x}$ to correspond to a normal surface it must satisfy the following two constraints.
First, in order for it to be possible to realize $F$ inside a given tetrahedron, there must be at most one quadrilateral disk type in each tetrahedron of $\mathcal{T}$.
Moreover, for two tetrahedra glued along a common $2$-face, the number of incident elementary disks on both sides must match up to form a surface. 
For a given arc type in a $2$-face of $\mathcal{T}$ and a tetrahedron on which this $2$-face lies one, there are exactly two disk types that contribute to that given arc type.
Hence, for each arc type of a $2$-face, the sum of the number of disk types that contribute to that arc type on both sides must be equal. 
For $\overrightarrow{x} = (x_1, x_2, \dots, x_{7t})$, the above constraints can be expressed as a system of \textit{matching equations}:
\begin{align*}
    &x_i + x_j = x_k + x_l\\
    &0 \leq x_i\;(0 \leq i \leq 7t)
\end{align*}
We have one equation per arc type in $\mathcal{T}^{(2)}$.     
All non-negative solutions to the matching equations form a linear cone $\mathcal{S}_\mathcal{T} \subset \mathbb{R}^{7t}$ called the \textit{normal solution space}.

\subsection{Computational Representation of Algebraic Field Elements}\label{field_appendix}
Recall that for a given algebraic field extension $\Q(\alpha)$, we can represent every element of the field uniquely as a $\Q$-linear combination of the basis $1, \alpha, \alpha^2,\dots,\alpha^{n-1}$
Thus, distinguishing elements of this field can be done computationally without any ambiguity.
The algebraic numbers $a$ and $b$ in $\Q(\alpha)$ can be represented by the degree $n-1$ polynomials $f_a(\alpha)$ and $f_b(\alpha)$.
Then the sum, product, and quotient of these numbers can be computed easily from these representations:
$f_{a\pm b}(\alpha)$ will be equal to $f_a(\alpha) \pm f_b(\alpha)$. 
$f_{ab}(\alpha)$ will be equal to the remainder of dividing $p(\alpha)$ by $f_a(\alpha)f_b(\alpha)$, and $f_{a/b}(\alpha)$ will be equal to the remainder after dividing $p(\alpha) + f_a(\alpha)$ by $f_b(\alpha)$.

\section{Proofs}\label{proofs}

\begin{proof}[Proof of Proposition~\ref{orientable_homology}]
    Note that the proof provided in \cite{Dunfield_2022} also applies to cusped finite-volume hyperbolic 3-manifolds by passing to the compact core.
\end{proof}

\begin{proof}[Proof of Lemma~\ref{tot_geo_essential}]
    Let $\alpha$ be a closed loop in $F$ corresponding to a non-trivial element in $\pi_1(F)$. We know that $\alpha$ is homotopic to a closed geodesic loop $\gamma$. Since $F$ is totally geodesic, then $\gamma$ is also a closed geodesic in $M$.
    Note that the universal cover of $M$ is $\mathbb{H}^3$ with covering map $p : \mathbb{H}^3 \rightarrow M$. Now, since $p$ is also a local isometry, $\gamma$ lifts to an embedded curve $\Tilde{\gamma}$ which is a non-trivial geodesic segment of some infinite geodesic ray. This implies that the lift corresponds to a non-trivial element of $\pslc$. Hence $\gamma$ corresponds to a non-trivial element of $\pi_1(M)$ and thus $\pi_1(F)$ injects into $\pi_1(M)$, proving the lemma.
\end{proof}

\begin{proof}[Proof of Lemma~\ref{fundamental_skeleton_surface}]
    We can simply construct $\mathcal{D}$ by taking the dual of $\mathcal{C}$.
    
    Take a spanning tree $\mathcal{T}$ of the 1-skeleton of $\mathcal{D}$. Quotient out the 1-skeleton of $\mathcal{D}$ by $\mathcal{T}$ to use that as the basepoint. The remaining edges in the 1-skeleton form a generating set with relators given by the 2-cells of $\mathcal{D}$.
\end{proof}

\begin{proof}[Proof of Lemma~\ref{fundamental_skeleton_manifold}]
    Take $M$ to be an ideally triangulated 3-manifold and cellulation $\mathcal{C}$. Construct an ordinary triangulation $\hat{\mathcal{C}}$ by adding vertices to the triangulation. 
    Then take $\mathcal{D}$ to be the 2-skeleton of the dual of $\hat{\mathcal{C}}$.
    Since $\mathcal{C}$ has no vertices, it will deformation retract onto $\mathcal{D}$ by blowing up from the vertices.
     
    Take a spanning tree $\mathcal{T}$ of the 1-skeleton of $\mathcal{D}$. Quotient out the 1-skeleton of $\mathcal{D}$ by $\mathcal{T}$ to use that as the basepoint. The remaining edges in the 1-skeleton form a generating set with relators given by the 2-cells.
\end{proof}

\begin{proof}[Proof of Lemma~\ref{skeletal_embedding_lemma}]
    The algorithm terminates because there are finitely many edges in $\mathcal{T}_F$ and $\mathcal{T}_M$.
    
    Any loop from the basepoint disk in $F$ is homotopic in $F$ to a cycle in the dual 1-skeleton of $F$. Let this cycle consist of the edges $e_1, \dots, e_n$. In the process above, for every $i$, the path from the end of $e_i$ back to the basepoint will be the reverse of the path from the basepoint to the beginning of $e_{i+1}$. These paths will cancel out giving a sequence of edges in the dual 1-skeleton of $M$ which is homotopic to the original loop. This then shows that this is the embedding map induced by inclusion from $F$ to $M$.
\end{proof}

\begin{proof}[Proof of Lemma~\ref{VolumeGenus}]
    Let $F$ be a connected embedded totally geodesic surface in $M$. Define $\overline{M}$ to be the 3-manifold obtained by cutting $M$ along $F$. 
    Observe that the Euler characteristic of $\partial\overline{M}$ is $2\chi(F)$. 
    By Theorem~4.2 of \cite{MIYAMOTO94}, we have $\frac{\vol \overline{M}}{\area \partial\overline{M}} \geq \mu_3(0).$
    Thus,
    \begin{align*}
       &\frac{\vol \overline{M}}{\area \partial\overline{M}} = \frac{\vol M}{2\area F} \geq \mu_3(0)\\
       \Rightarrow \quad &\vol M \geq 2 \area F \cdot \mu_3(0) = 2 \cdot \lvert 2\pi\chi(F) \rvert \cdot \mu_3(0).
    \end{align*}
    Thus, for orientable $F$, we get $\vol M \geq 2 \mu_3(0) \area  F \gtrapprox 7.3277$ and for non-orientable $F$, we get $\vol M \geq 2 \mu_3(0) \area  F \gtrapprox 3.66385$.\\
    For the second part of the lemma, let $N$ be a 3-manifold containing a totally geodesic surface. Hence, we get
    \begin{align*}
       &\area F \leq \frac{\vol M}{2 \mu_3(0)}\\
       \Rightarrow \quad & \lvert 2\pi\chi(F) \rvert \leq \frac{\vol M}{2 \mu_3(0)}\\
       \Rightarrow \quad & \lvert \chi(F) \rvert \leq \frac{\vol M}{4 \pi \mu_3(0)}.
       \end{align*}
\end{proof}

\begin{proof}[Proof of Lemma~\ref{finite_isotopy_lemma}]
By Theorem~\ref{allsurface_iso_normal} it suffices to examine all normal surfaces which correspond to an admissible integral vector~\cite{Haken}.
Using Algorithms 3.1 and 3.2 in~\cite{Burton_2013}, we can enumerate all fundamental normal surfaces in $M$. All normal surfaces can be found as a linear combination of fundamental normal surfaces.
Observe Theorem 4.3 of \cite{Lackenby2007AnAT} states that given a genus $g$, there are only finitely many normal surfaces in $M$ which have genus at most $g$.
This implies that there are only finitely many normal surfaces up to a certain Euler characteristic since any disconnected surface can be viewed as the union of its connected components.
\end{proof}

\section{Additional Implementation Details}\label{implementation_details}
The algorithm was run on 200 jobs simultaneously every 7 days walltime. 
Jobs were run on one of three partitions with the following specifications:
\begin{enumerate}[a)]
    \item The first partition has 20 nodes each with two Xeon CPU E5-2640 0 @ 2.50Ghz, which has six cores. 
    \item The second partition has 16 nodes each with two Xeon CPU E5520 @ 2.27GHz, which has four cores.
    \item The third partition has 4 nodes each with Xeon CPU E5-2690 v3 @ 2.60GHz.
\end{enumerate}

Some algorithms that check for essential surfaces are mentioned here. 
One such method presented in Regina aims to detect a compressing disk by cutting the 3-manifold along the given surface, retriangulating, then searching for a compressing disk in each component of the cut-open triangulation.
This technique is based on the work of Jaco, Oertel which itself is considered impractical since it may take double exponential time.
Practical implementations are explored in~\cite{Burton2012ComputingCE}.

Another idea presented by David Letscher aims to detect certain compressing disks in normal surfaces: let $e$ be an edge in the triangulation of a manifold $M$ and $F$ a normal surface in $M$.
If every tetrahedron glued to $e$ contains a quadrilateral of $F$ that separates $e$ from its unique disjoint edge in the tetrahedron then these quadrilaterals form a tube which contains an obvious disk $D$ embedded in $F$.
This implies that either $D$ is a compressing disk and $F$ is compressible or that $F$ is not least weight, meaning that it is isotopic to another normal surface that has strictly fewer intersections with the 1-skeleton of the manifold.
Note that as an algorithm this method is efficient in finding potential compressible surfaces of a specific shape, but does not always definitively determine whether or not a surface is essential.

\section{Further Sanity Checks}\label{sanity_checks}

In order to check that our implementation of Algorithm~\ref{skeletal_embedding} was correct, we also created an algorithm which not only finds the generators of the fundamental group of the surface, but additionally returns a set of relations in order to get a full presentation of the fundamental group of the surface. We did this by identifying disks in our normal surface and taking the generators of the boundary of these disks in order.

\begin{figure}
		\centering
		\includegraphics[scale = 0.4]{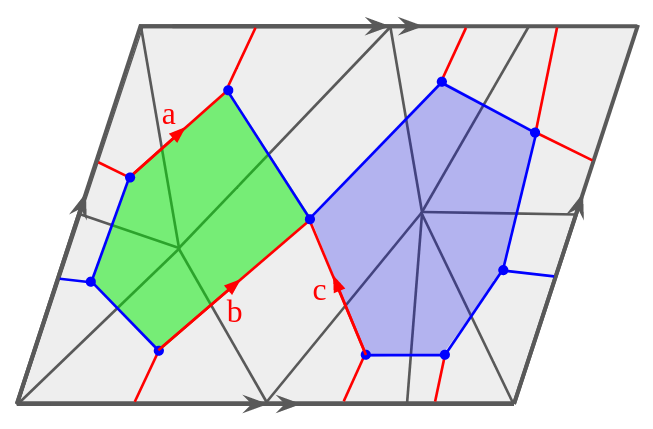}
		\caption{Two disks in a cellulation of a surface which give some relations. The green disk gives the relation $ab^{-1}$ and the blue one gives the relation $c$.
        }
		\label{relation_disks}
\end{figure}

Finding these relations has several benefits. Firstly, the fundamental group of an orientable genus $g$ surface can always be specified by a presentation with $2g$ generators $a_1,\dots,a_g, b_1,\dots, b_g$ and a single relation of the form $\prod_{i=1}^g a_ib_ia_i^{-1}b_i^{-1}$ \cite{EppsteinFundamentalGroup}. For non-orientable surfaces, there is a similar presentation with $g$ generators $c_1,\dots,c_g$ and the single relation is $\prod_{i=1}^g c_i^2$. This comes from the fact that every surface can be canonically constructed as a gluing of a polygon with an even number of sides. As a CW decomposition, this gives us a CW structure with a single 0-cell, $2g$ 1-cells in the orientable case ($g$ 1-cells for non-orientable surfaces), and a single 2-cell.

Given the fundamental group of a 2-complex, in order to check it is the fundamental group of a surface, all that needs to be shown is that it is isomorphic to a group of the above type. We can do this in the following steps, using a procedure laid out in remark 6.14 of \cite{dinakar_nathan}.
\begin{enumerate}[a)]
    \item Simplify the presentation so that it has a single relation.
    If this cannot be done, then it cannot be the presentation of a surface.
    \item This presentation gives us a formula for how to construct the CW complex of a surface. We take a single 0-cell, attach a 1-cell for each generator, and then the relation word tells us how to attach the lone 2-cell.
    \item What remains to check is that this is topologically a surface. That is, around each point in our complex, we have a neighborhood that looks like $\R^2$. This can be done in three steps:
    \begin{enumerate}[i)]
        \item We know all the interior points of the 2-cell have this property, so we are done there.
        \item For any point on one of the 1-cells, we just need to check that the 2-cell glues to it twice. This amounts to checking that for each generator $g$ that appears in the relator word, $g$ and $g^{-1}$ together appear exactly twice. For example, the word $ab^{-1}ab$ would be a valid relator but $aba^{-1}b^{-1}ab$ would not be valid since you have an $a$ and two $a^{-1}$ in the word. For an orientable surface, we need something stronger, that $g$ and $g^{-1}$ appear exactly once \textit{each}. Here, $ab^{-1}ab$ would not be a valid word since $a$ appears twice and $a^{-1}$ does not appear at all.
        \item We check that the neighborhood around the 0-cell is homeomorphic to a disk. That is, we must check that the corners of the 2-cell all glue up correctly to the 0-cell. One way to do this is by checking that as you glue the opposite edges of the disk corresponding to the 1-cells, all of the corners of the disk get identified.
    \end{enumerate}
\end{enumerate}

We performed this check on many examples to ensure that the code which gives the generating set was working correctly.

Additionally, we can check our implementation of Algorithm~\ref{skeletal_embedding} by applying the resulting embedding map on the relations we find in our surface. They should always end up being trivial, as loops which bound disks in the surface also bound disks in the manifold. We also applied this check on examples to ensure that our embedding code was working correctly as well.

\section{Code Availability}\label{code_availability}

The main portion of the code used in this paper at this Harvard Dataverse link~\cite{DVN/3YGSTI_2024}:

\href{https://doi.org/10.7910/DVN/3YGSTI}{https://doi.org/10.7910/DVN/3YGSTI}.

In order to run this code, you will need to install Sagemath, SnapPy, and Regina on your computer.
To run this code on manifolds with multiple cusps, the default version of Regina cannot be used as it does not enumerate surfaces in triangulations with multiple cusps. Instead we used Nathan Dunfield's fork of regina-normal/regina, specifically the multicusp-closed branch. This adds in the ability to enumerate surfaces for manifolds with multiple cusps.
\end{document}